\documentclass[11pt]{amsart}
\usepackage[margin=2.7cm, marginpar=1cm]{geometry}

\usepackage[T1]{fontenc}
\usepackage{comment}
\usepackage{appendix}
\usepackage{graphicx}
\usepackage{pdfpages}
\usepackage{nicematrix}
\newtheorem{theorem}{Theorem}[section]
\theoremstyle{definition}

\theoremstyle{remark}
\newtheorem{remark}[theorem]{Remark}
\usepackage{subcaption}
\usepackage{hyperref}
\usepackage[font=small]{caption}

\usepackage{epic}
\usepackage{import}
\usepackage{pdfpages}

\theoremstyle{plain}
\newtheorem{thm}{Theorem}
\newtheorem{lem}[thm]{Lemma}
\newtheorem{prop}[thm]{Proposition}
\newtheorem{cor}[thm]{Corollary}

\newtheorem{defi}[thm]{Definition}
\theoremstyle{definition}
\newtheorem{exa}[thm]{Example}

\usepackage{float}
\title{A bound on the equivariant unknotting number}
\author{Sarah Zampa}

\DeclareMathOperator{\sign}{sign}
\DeclareMathOperator{\rank}{rank}
\def\sigmaeq{\widetilde{\sigma}}
\def\ueq{\widetilde{u}}
\def\ueqa{\widetilde{u}_A}
\def\ueqb{\widetilde{u}_B}
\def\ueqc{\widetilde{u}_C}

\DeclareMathOperator{\Fix}{Fix}

\begin{document}

\begin{abstract}
We study how the equivariant signature of strongly invertible knots changes when one of the Boyle-Chen equivariant unknotting moves is applied.
%We prove that the absolute value of the equivariant signature is bounded above by three times the type A equivariant unknotting number, by twice the type C equivariant unknotting number for some $(1,2)$-knots, and, for admissible diagrams, by twice the type B equivariant unknotting number.
It follows form our results that the absolute value of the equivariant signature introduced by Alfieri-Boyle gives a lower bound to three times the equivariant unknotting number.
%These results provide analogues in the equivariant setting of the classical bound relating the knot signature to the unknotting number.
\end{abstract}
{{\maketitle}}

\vspace{-2.5mm}

\section{Introduction}

A symmetric knot $(K,\rho)$ is a knot together with an orientation-preserving homeomorphism $\rho : S^3\rightarrow S^3$ which maps $K$ to itself. We say that $K$ is a \emph{strongly invertible knot} if it is a symmetric knot such that the involution $\rho$ reverses the orientation on $K$. In this case, the fixed point set of the involution intersects $K$ in two points, and there is an axis of symmetry which we will denote by $S$.
There has been a recent interest in strongly invertible knots, through their connections with $4$-manifold topology and strong corks (see \cite{DMS23} and references therein). There are still many questions about strongly invertible knots from a purely knot theoretical viewpoint.

Recently, in \cite{BC24} Boyle and Chen considered an equivariant version of the unknotting number. This is defined looking at   regular equivariant homotopies from $K$ to the unknot. The \emph{total equivariant unknotting number} of a strongly invertible knot $K$, denoted by $\ueq(K)$, is then defined as the  minimum number of transverse self-intersections appearing in any such homotopy. These are partitioned into three categories: 
\begin{itemize}
    \item type A: when two transverse double points happen away from the axis of symmetry, and in such a way that they are interchanged by the involution $\rho$;
    \item type B: when there is one  transverse double point lying on the axis;
    \item type C: the equivariant homotopy contracts a sub-arc of the axis of symmetry, hence creating two new crossings on the projection of the knot.  
\end{itemize}
See Figure \ref{fig:movetypes} for a pictorial depiction of the various cases.
\begin{figure}[ht!]
    \def\svgwidth{0,5\columnwidth}
    \centering
    %
    %% Creator: Inkscape 1.2.2 (732a01da63, 2022-12-09), www.inkscape.org
%% Accompanies image file 'movetypes.pdf' (pdf, eps, ps)
\begingroup%
  \makeatletter%
  \providecommand\color[2][]{%
    \errmessage{(Inkscape) Color is used for the text in Inkscape, but the package 'color.sty' is not loaded}%
    \renewcommand\color[2][]{}%
  }%
  \providecommand\transparent[1]{%
    \errmessage{(Inkscape) Transparency is used (non-zero) for the text in Inkscape, but the package 'transparent.sty' is not loaded}%
    \renewcommand\transparent[1]{}%
  }%
  \providecommand\rotatebox[2]{#2}%
  \newcommand*\fsize{\dimexpr\f@size pt\relax}%
  \newcommand*\lineheight[1]{\fontsize{\fsize}{#1\fsize}\selectfont}%
  \ifx\svgwidth\undefined%
    \setlength{\unitlength}{562.22942222bp}%
    \ifx\svgscale\undefined%
      \relax%
    \else%
      \setlength{\unitlength}{\unitlength * \real{\svgscale}}%
    \fi%
  \else%
    \setlength{\unitlength}{\svgwidth}%
  \fi%
  \global\let\svgwidth\undefined%
  \global\let\svgscale\undefined%
  \makeatother%
  \begin{picture}(1,1.14915256)%
    \lineheight{1}%
    \setlength\tabcolsep{0pt}%
    \put(0,0){\includegraphics[width=\unitlength,page=1]{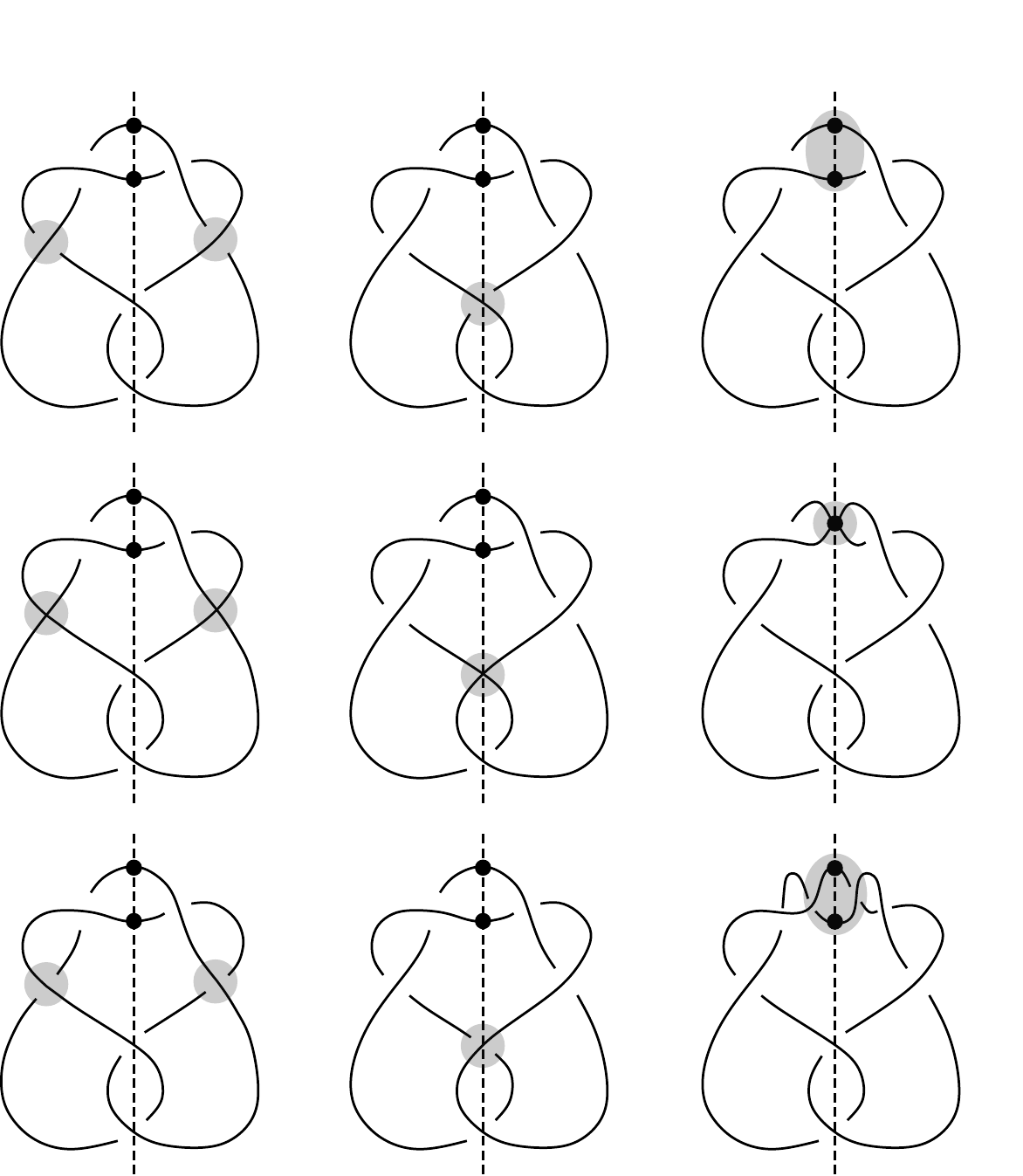}}%
    \put(0.10577866,1.10621122){\makebox(0,0)[lt]{\lineheight{1.25}\smash{\begin{tabular}[t]{l}\LARGE $A$\end{tabular}}}}%
    \put(0.44727642,1.10621122){\makebox(0,0)[lt]{\lineheight{1.25}\smash{\begin{tabular}[t]{l}\LARGE $B$\end{tabular}}}}%
    \put(0.78877424,1.10621122){\makebox(0,0)[lt]{\lineheight{1.25}\smash{\begin{tabular}[t]{l}\LARGE $C$\end{tabular}}}}%
  \end{picture}%
\endgroup%

    \caption{Examples of type A, B and C moves applied to $6_1$; we can see that $\ueqa(6_1)=\ueqb(6_1)=1$.}
    \label{fig:movetypes}
\end{figure}
Note that by considering only one type of move, one can also define the type X unknotting number $\ueq_X$, for $\text{X}\in\{\text{A},\text{B},\text{C}\}$. A knot may fail to be unknotted using only type X moves, in which case we set $\ueq_X(K)=\infty$.

In the case of the classical unknotting number, one can show that the signature satisfies the well-known bound $\lvert\sigma(K)\rvert\leq2u(K)$ (first shown in \cite{M65}). Here, we will focus on \emph{directed} strongly invertible knots:
if $S$ denotes the axis of symmetry of a strongly invertible knot $K$, the set $S\setminus K=h\cup h'$ splits into two arcs $h$ and $h'$. A \emph{direction} of $K$ is then a choice of one of the two arcs.

In \cite{AB24}, Alfieri and Boyle introduce an invariant $\sigmaeq(K)$ of directed strongly invertible knots called the \emph{equivariant signature}, based on the $G$-signature formula of Atiyah-Singer \cite{AS68}. In this paper, we show that for each unknotting type, a similar type of bound as mentioned in the previous paragraph, occurs for the equivariant case.

\begin{thm}
    Let $K$ be a directed strongly invertible knot, and consider a homotopy from $K$ to the unknot which never crosses the direction. Then, $\lvert\sigmaeq(K)\rvert/3$ gives a lower bound to the number of transverse self-intersections.
\end{thm}

This paper is organized as follows: we will first consider double-points of type B as the proof is the most straightforward. For this type, there is a caveat due to the fact that the equivariant signature (like many other invariants of strongly invertible knots) is an invariant of \emph{directed} strongly invertible knots.
Indeed, a type B double point can either lie on the direction, in which case we call it a \emph{directed type B} move, or it lies on the other arc.

\begin{thm}\label{mainthmB}
    If $K$ and $K'$ are two directed strongly invertible knots which differ by a directed type B move, then $\sigmaeq(K)-2\leq\sigmaeq(K')\leq\sigmaeq(K)+2$.
\end{thm}

It is still an open question whether every strongly invertible knot can be unknotted using only type B moves. This question is closely related to Nakanishi's $4$-move conjecture \cite{NS87}, which appears in Kirby's list of problems (see \cite[1.59(3)(a)]{K97}). Indeed, a positive answer would imply a positive resolution of the $4$-move conjecture (see \cite[Corollary 5.3]{BC24}).

\begin{cor}
    For a regular equivariant homotopy from a directed strongly invertible knot $K$ to the unknot where there are only directed type B self-intersections, the number of critical points is bounded below by $\lvert\sigmaeq(K)\rvert/2$.
\end{cor}

Considering now type C moves, we have the following result.

\begin{thm}\label{mainthmC}
    If $K$ and $K'$ are two directed strongly invertible knots which differ by a type C move, then $\sigmaeq(K)-2\leq\sigmaeq(K')\leq\sigmaeq(K)+2$.
\end{thm}

It has been shown in \cite[Theorem 1.8]{BC24} that the only knots which can be unknotted using only type C moves are $(1,2)$-knots where their axis of symmetry coincides with the core of one of the handlebodies in the $(1,2)$-decomposition. We get an immediate corollary:

\begin{cor}
    Let $K$ be a directed strongly invertible knot. If $K$ is a $(1,2)$-knots such that the axis of symmetry coincides with the core of one the handlebodies in the $(1,2)$-decomposition, then $\lvert\sigmaeq(K)\rvert\leq2\ueqc(K)$.
\end{cor}

Finally for the generic case of type A double points, we obtain a slightly different result.

\begin{thm}\label{mainthmA}
    If $K$ and $K'$ are two directed strongly invertible knots which differ by a type A move, then $\sigmaeq(K)-6\leq\sigmaeq(K')\leq\sigmaeq(K)+6$.
\end{thm}
It has been shown by Boyle and Chen that every strongly invertible knot $K$ admits an equivariant homotopy with only type $A$ double points which unknots $K$. Using this fact, we can consider the type $A$ unknotting number $\ueq_A$ and derive the following result.
\begin{cor}
    If $K$ is a directed strongly invertible knot, then $\lvert\sigmaeq(K)\rvert\leq3\ueqa(K)$.
\end{cor}

%Having established those results, there are natural questions that arise.
%We have established the result for type B after fixing a specific direction out of the two. One can first ask if the difference of equivariant signatures when considering either direction is arbitrarily large, or can it be bounded by an invariant?
%If that was the case, then the equivariant unknotting number 

%\textcolor{blue}{The open questions/conjectures. For the two directions, is the difference of signatures bounded? // For type C, how to go from K to K' (connected components; we know the one for unknot, what about the rest)? "the type B crossing change for non-admissible diagrams (Is it true that the jump in signature in the general type B case can be arbitrarily large?)"}

\vspace{5mm}
\noindent\textbf{Acknowledgements:} The author is deeply grateful to Antonio Alfieri for his invaluable help and guidance throughout the course of this project.

\section{Background material}

\subsection{Admissible diagrams}\label{section:admissible}
We first address the admissibility condition which will be relevant in the definition of the correction term.
\begin{defi}
    If $h$ is the chosen direction, then $F$ is an \textbf{admissible} surface for $K$ if $\partial F=K$, and it satisfies that $h\subset F$ and $h'\cap F=\emptyset$.
\end{defi}
We can naturally extend the definition of admissibility to the diagrams themselves, by requiring that there are no crossings along $h'$. This is equivalent to the checkerboard surface containing $h$ being admissible.

\begin{prop}\label{prop:admissible}
    A strongly invertible knot always admits an admissible diagram.
\end{prop}
\begin{proof}
    If a diagram of a strongly invertible knot is not admissible, then there are some crossings along the arc $h'$. We can always ``push'' those along the axis of symmetry and out of the directed arc, as depicted in Figure \ref{fig:admissiblediagram} (here, the directed arc is the non-compact arc). 
    \begin{figure}[ht!]
    \def\svgwidth{0,5\columnwidth}
    \centering
    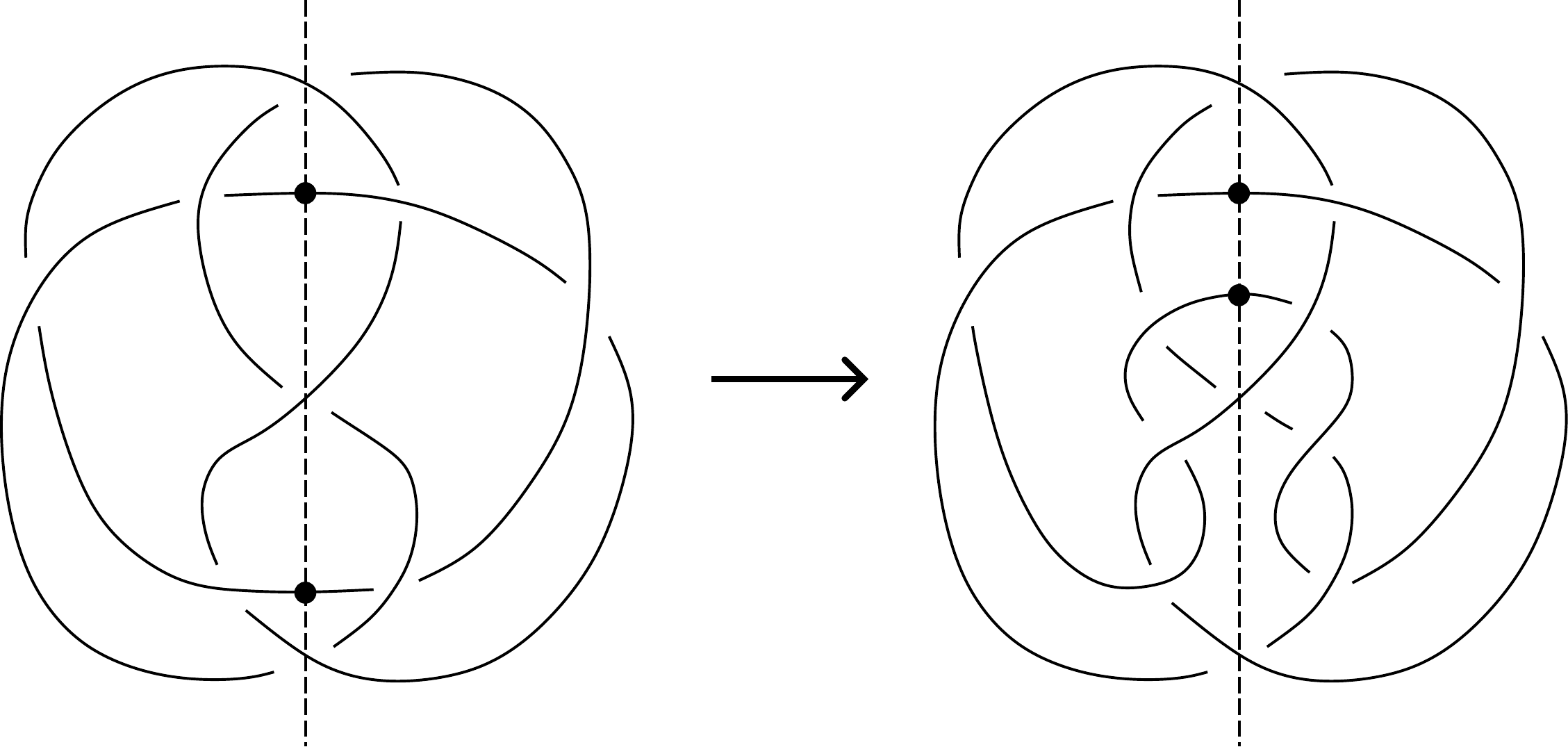

    \caption{A move showing that one can always make a diagram admissible.}
    \label{fig:admissiblediagram}
\end{figure}
\end{proof}

\begin{remark}
    %For crossing changes of type A and C, we can perform this previously described move before or after performing the crossing change, and obtain the same knot. However, for a type B crossing change, performing this move before or after the crossing change will not result in the same knot. This is what leads us to consider only admissible diagrams in the type B case.
    Note that the moves performed in the procedure underlying Proposition \ref{prop:admissible} commute with  unknotting moves of type A or C.
    %By performing this previously described ``push'' along the axis either before or after a type A or C crossing change, we will obtain the same knot.
    However, this is not the case for a type B crossing change since the push moves may cross the arc contracted by the regular homotopy, resulting in a different knot after the crossing change operation.
    %Depending on the order in which we apply the push moves and the unknotting type B move, we might end up with different knots.
    This is what leads us to consider only admissible diagrams in the type B case.
\end{remark}

\subsection{The equivariant signature and the Goeritz matrix}
Given a spanning surface $F$ for a knot $K$, Gordon and Litherland \cite{GL78} associate a bilinear pairing
\[
\theta \colon H_1(F;\mathbb{Z}) \times H_1(F;\mathbb{Z}) \to \mathbb{Z}.
\]
The signature of $\theta+\theta^T$ depends on the choice of surface $F$ and therefore is not a knot invariant. However, Gordon and Litherland \cite[Theorem 6]{GL78} show that after adding an explicit correction term, one obtains an invariant of the knot.
When the surface $F$ arises from a checkerboard coloring of a knot diagram, the pairing $(H_1(F),\theta)$ admits a purely diagrammatic description in terms of the \emph{Goeritz matrix}. We briefly recall this construction.

Consider a checkerboard coloring of the knot diagram, and label by $X_0,X_1,\ldots,X_n$ the white regions. For each crossing $c$, assign a number $\eta(c)$ which is either $+1$ or $-1$ as given by Figure \ref{fig:crossingsign}.
\begin{figure}[h!]
    \centering
    \includegraphics[width=0.4\textwidth]{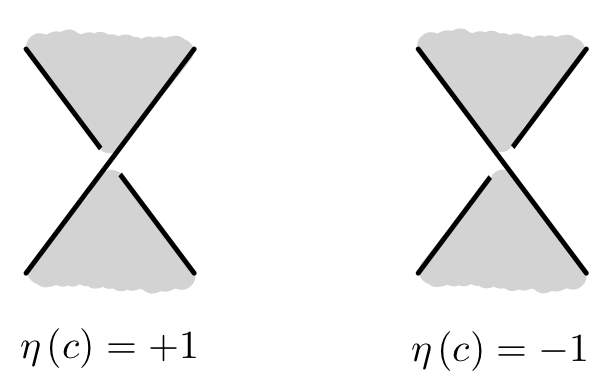}
    \caption{The incidence number $\eta(c)=\pm1$ associated to a crossing $c$.}
    \label{fig:crossingsign}
\end{figure}
Let $S_{ij}$ denote the set of crossings incident to both $X_i$ and $X_j$. We can then define a matrix $G'=(g_{ij})_{0\leq i,j\leq n}$ such that
$$g_{ij}=\begin{cases}
-\sum\limits_{c\in S_{ij}}\eta(c) \quad\text{if } i\neq j,\\
\sum\limits_{0\leq k\leq n, k\neq i}g_{ik}\quad\text{ otherwise}.
\end{cases}$$

As shown in \cite{GL78}, the symmetric form represented by $G'$ computes the Gordon-Litherland pairing $\theta+\theta^T$ on $H_1(F;\mathbb{Z})$, after identifying $H_1(F)$ with a free abelian group $V$ generated by $X_0, \dots, X_n$ modulo the relation $X_0+\ldots+X_n=0$.

\begin{defi}
    The Goeritz matrix $G$ is the matrix obtained from $G'$ by deleting the $i$-th row and column for some $0\leq i\leq n$.
\end{defi}

Note that in particular $\sign(G)= \sign(\theta+\theta^T)$. In our case, where we consider symmetric knot diagrams, we delete the row and column corresponding to the unique white region fixed by the involution $\rho$, namely the region containing the arc $h'$ between the two fixed points $\Fix(\rho)\cap K$.

%%%%%%%%%%%%%%%%%%%%%%%%%%%%%%%%%%%%%%

%Now suppose that $K$ is a directed strongly invertible knot with involution $\rho$. In \cite{AB24}, Alfieri and Boyle define an \emph{equivariant signature} refining the Gordon-Litherland construction. As in the classical case, an additional correction term is required to obtain an invariant independent of the chosen equivariant spanning surface. (The correction as in the Gordon-Litherland construction is suggested by the Atitah-Singer $G$-signature theorem.)
Now, in a similar fashion as Gordon and Litherland in \cite{GL78}, Alfieri and Boyle \cite{AB24} defined a certain correction term $e(D)$ so that the equivariant signature will not depend on the chosen equivariant (possibly non-orientable) surface $F$ having $K$ as boundary. In \cite{AB24}, Alfieri and Boyle define this correction term (which is derived from the $G$-signature formula of Atiyah-Singer \cite{AS68}), and find an explicit diagrammatic definition of it for admissible diagrams which we describe now.

Consider a diagram $D$ of $K$, and consider the two arcs $D\setminus\Fix(\rho)=a\cup b$ mapped to each other by $\rho$; we orient both arcs such that their orientation matches under $\rho$. For a crossing $c$, we assign a sign $\epsilon(c)=\pm 1$ induced by the orientations on $a$ and $b$:
\begin{defi}
    If $D$ is an admissible diagram for $K$, then with respect to the checkerboard surface, the correction term can be defined as
    \[
    e(D)=-\sum_{c\in a\cap b, c\notin h} \epsilon(c).
    \]
\end{defi}

\begin{comment}
    The involution $\rho$ induces an involution on $H_1(F;\mathbb{Z})$, yielding a decomposition
\[
H_1(F;\mathbb{Z}) = E_+ \oplus E_-,
\]
where $E_\pm$ are the $\pm1$ eigenspaces. Denoting by $\sigma$ the signature of a symmetric bilinear form, the equivariant signature is defined as follows.
\end{comment}
In the future, we will refer to a crossing $c$ as \emph{bicolored} if $c\in a\cap b$, or \emph{unicolored} otherwise.
Now, $\rho$ induces an involution on the first homology group of the surface $F$, such that the latter splits into a direct sum $E_+\oplus E_-$ of the $+1$ and $-1$ eigenspaces of the involution. Denoting by $\sigma$ the usual signature, the equivariant signature is defined as follows.

\begin{defi}[Alfieri-Boyle \cite{AB24}]
    For a directed strongly invertible knot $K$, the equivariant signature is defined as
    $$\sigmaeq(K)=\sigma((\theta+\theta^T)\lvert_{E_+})-\sigma((\theta+\theta^T)\lvert_{E_-})-e(D).$$
\end{defi}
%Note that the equivariant signature does not depend on the choice of diagram of the knot $K$.

Finally we describe a basis realizing the splitting $E_+\oplus E_-$ in presence of a  symmetric diagram. Label by $a_i$ and $a'_i$, $1\leq i\leq n$,  the white regions $X_1, \dots, X_n$ of the checkerboard coloring so that that $\rho(a_i)=-a'_i$.
%(The unique white region fixed by $\rho$ has been removed in the definition of the Goeritz matrix.)
We will write the Goeritz pairing in the basis $(a_1,\ldots,a_n,a'_1,\ldots,a'_n)$. Note that the change of basis   
\[(a_1,\ldots,a_n,a'_1,\ldots,a'_n) \to (a_1-a_1',\ldots,a_n-a_n',a_1+a'_1,\ldots,a_n+a'_n)\]
gives a block decomposition 
$$\begin{pmatrix}
        M^+ & 0\\
        0 & M^-
    \end{pmatrix}$$
with $M^+$ and $M^-$ representing the restriction of the Goeritz form on the eigenspaces $E_+$ and $E_-$ respectively.

\begin{comment}
The Goeritz pairing defines a symmetric bilinear form
\[
\mathcal{G}\colon V\times V\to \mathbb{Z},
\]
whose matrix with respect to the ordered basis $(a_1,\ldots,a_n,a'_1,\ldots,a'_n)$ is precisely the Goeritz matrix $G$. Explicitly, the elements $g_{ij}$ of the Goeritz matrix can be written as $\mathcal{G}(a_i,a_j)$ if $1\leq i,j\leq n$, and as $\mathcal{G}(a_i,a'_j)$ if $1\leq i\leq n$ and $n+1\leq j\leq 2n$; the matrix is symmetric so the rest of the elements are given accordingly.

    the elements $g_{ij}$ of the Goeritz matrix can be written as $\mathcal{G}(a_i,a_j)$ if $1\leq i,j\leq n$, and as $\mathcal{G}(a_i,a'_j)$ if $1\leq i\leq n$ and $n+1\leq j\leq 2n$; the matrix is symmetric so the rest of the elements are given accordingly.
\end{comment}

\begin{exa}
    In Figure \ref{fig:movetypes}, a Goeritz matrix for the given graph of $6_1$ is
    \[
M=\begin{bNiceArray}[first-row,first-col]{ccccc}
    & a_1 & a_2 & a'_1 & a'_2 \\
    a_1 & -3 & 1 & 2 & 0 \\
    a_2 & 1 & -2 & 0 & 0 \\
    a'_1 & 2 & 0 & -3 & 1 \\
    a'_2 & 0 & 0 & 1 & -2 \\
\end{bNiceArray}
\]
and the matrices of the $+1$ and $-1$ eigenspaces over $\mathbb{R}$ are
\[
M^+=\begin{bNiceArray}[first-row,first-col]{ccc}
    & a_1-a'_1 & a_2-a'_2 \\
    a_1-a'_1 & -10 & 2 \\
    a_2-a'_2 & 2 & -4 \\
\end{bNiceArray}, \quad
M^-=\begin{bNiceArray}[first-row,first-col]{ccc}
    & a_1+a'_1 & a_2+a'_2 \\
    a_1+a'_1 & -2 & 2 \\
    a_2+a'_2 & 2 & -4 \\
\end{bNiceArray}.
\]
Since the projections of the two arcs $a$ and $b$ only meet on the axis of symmetry, the correction term is $e(D)=0$ here, and we find that $\sigmaeq(6_1)=0-2=-2$.
\end{exa}

\begin{remark}
    Later, we shall use the fact that both matrices $M^+$ and $M^-$ have \emph{non-zero} determinant. Indeed, we can perform a change of basis to give $M$ as the block matrix
    $$\begin{pmatrix}
        M^+ & 0\\
        0 & M^-
    \end{pmatrix}=\begin{pmatrix}
        I & -I\\
        I & I
    \end{pmatrix} M\begin{pmatrix}
        I & I\\
        -I & I
    \end{pmatrix}$$
    where $I$ denotes the identity matrix of size $n$ (if $M$ is of size $2n$). Taking the determinant, one concludes that $\det M^+ \cdot \det M^-= 4^n \det M$. Since  $|\det M|=\det(K)$ and for a  knot $\det(K)\not=0$ (being an \emph{odd number}) one concludes that both $\det M^+\not=0$, and $\det M^-\not=0$.
\end{remark}

\section{Proof of Theorem \ref{mainthmB}}

We need to understand how the signatures of $M^+$ and $M^-$, the matrices of the restriction of the Seifert form over the $+1$ and $-1$ eigenspaces of the involution, change after modifying one entry from those matrices. The main tool here will be the following theorem on symmetric matrices:
\begin{thm}[Jones \cite{J50}]\label{thm:jones}
    Let $M=(m_{ij})$ be a symmetric $n\times n$ matrix over $\mathbb{R}$. There is a sequence of principal submatrices $\Delta_0=1, \Delta_1, \Delta_2, \ldots$ such that $\Delta_i$ is a principal minor of $\Delta_{i+1}$, and no consecutive $ \Delta_i, \Delta_{i+1}$ are both singular for $i<\rank(M)$ (this is called a \textbf{$\sigma$-series}). Additionally, for any such series, the signature of $M$ can be computed as
    \[
    \sigma(M)=\sum_{i=1}^n \sign(\det(\Delta_{i-1})\det(\Delta_i)).
    \]
\end{thm}
Note that in the following, we will always work will full-rank matrices.

Now, we will start by going through how the quantity $e(D)$ changes after the different move types.
\begin{lem}
    The correction term $e(D)$ follows the following pattern:
    \begin{itemize}
        \item after a type A move: it either remains unchanged, or changes by $\pm4$;
        \item after a type B move: it remains the same;
        \item after a type C move: it either remains unchanged, or changes by $\pm2$.
    \end{itemize}
\end{lem}
\begin{proof}
    A type A move is applied to two simultaneous crossings, which are mapped to each other by the involution $\rho$. In the case of unicolored crossings, the correction term is left unchanged. For bicolored crossings, notice that both crossings must have the same sign. If they were positive crossings before the move, they become negative crossings afterwards, and so the correction term changes by $+4$; for the opposite direction, it changes by $-4$.

    A type B move modifies a crossing on the axis of symmetry, and so it is not counted in the sum $e(D)$.

    Finally, a type C move creates two new crossings. Either they are both unicolored, in which case the correction term does not change, or they are bicolored, which changes the correction term by $\pm2$ depending on the sign of the new crossings.
\end{proof}

We are now ready to address the main part of Theorem \ref{mainthmB}.

We will show that in the type B case, the difference of signatures is either zero or $\pm2$.
Let $M_B$ be the Goeritz matrix associated to the diagram $D_B$ obtained by performing a type B crossing change on $D$. If the crossing change is performed on a crossing between the regions $a_n$ and $a'_n$, and if the original matrix associated to the diagram is
\[
M=\begin{bNiceArray}[first-row,first-col]{ccccccc}
    & a_1 & \cdots & a_n & a'_1 & \cdots & a'_n \\
    a_1 & & & & & & \\
    \vdots & & & & & & \\
    a_n & & & \alpha & & & \beta \\
    a'_1 & & & & & & \\
    \vdots & & & & & & \\
    a'_n & & & \beta & & & \alpha \\
\end{bNiceArray},
\]
then $M_B$ can be given as
\[
M_B=\begin{bNiceArray}[first-row,first-col]{ccccccc}
    & a_1 & \cdots & a_n & a'_1 & \cdots & a'_n \\
    a_1 & & & & & & \\
    \vdots & & & & & & \\
    a_n & & & \alpha\pm2 & & & \beta\mp2 \\
    a'_1 & & & & & & \\
    \vdots & & & & & & \\
    a'_n & & & \beta\mp2 & & & \alpha\pm2 \\
\end{bNiceArray},
\]
where the non-explicit coefficients are the same in both matrices.

The $+1$ and $-1$ eigenspaces of $\rho$ are in turn given by
\[
M^+=\begin{bNiceArray}[first-row,first-col]{cccc}
    & a_1-a'_1 & \cdots & a_n-a'_n \\
    a_1-a'_1 & & & \\
    \vdots & & P & \\
    a_n-a'_n & & & 2(\alpha-\beta) \\
\end{bNiceArray}\quad M^-=\begin{bNiceArray}[first-row,first-col]{cccc}
    & a_1+a'_1 & \cdots & a_n+a'_n \\
    a_1+a'_1 & & & \\
    \vdots & & Q & \\
    a_n+a'_n & & & 2(\alpha+\beta) \\
\end{bNiceArray}
\]

and

\[
M^+_B=\begin{bNiceArray}[first-row,first-col]{cccc}
    & a_1-a'_1 & \cdots & a_n-a'_n \\
    a_1-a'_1 & & & \\
    \vdots & & P & \\
    a_n-a'_n & & & 2(\alpha-\beta)\pm8 \\
\end{bNiceArray}\quad M^-_B=\begin{bNiceArray}[first-row,first-col]{cccc}
    & a_1+a'_1 & \cdots & a_n+a'_n \\
    a_1+a'_1 & & & \\
    \vdots & & Q & \\
    a_n+a'_n & & & 2(\alpha+\beta) \\
\end{bNiceArray}
\]
where $P$ and $Q$ are symmetric $(n-1)\times(n-1)$ matrices.

We will now use Theorem \ref{thm:jones} to compute the necessary signatures.
Note that $M^-=M^-_B$, such that we only need to find $\sigma$-series for $M^+$ and $M^+_B$; we can choose those as follow. We first pick $\Delta_1,\ldots,\Delta_{n-1}$ to be a $\sigma$-series for the matrix $P$. Then, $\Delta_1,\ldots,\Delta_{n-1},\Delta_n=M^+$ is a $\sigma$-series for $M^+$, and $\Delta_1,\ldots,\Delta_{n-1},\Delta^B_n=M^+_B$ is a $\sigma$-series for $M^+_B$. This is a proper choice since $M^+$ and $M^+_B$ are non-singular. This yields
\begin{align*}
    \sigma(M^+)&= \sum_{i=1}^{n-1} \sign\left(\det(\Delta_{i-1})\det(\Delta_i)\right) + \sign\left(\det(\Delta_{n-1})\det(M^+)\right)\\
    \sigma(M^+_B)&= \sum_{i=1}^{n-1} \sign\left(\det(\Delta_{i-1})\det(\Delta_i)\right) + \sign\left(\det(\Delta_{n-1})\det(M^+_B)\right)
\end{align*}

Now, having that $\Delta_{n-1}=P$, we have that the difference of signatures is given by
\begin{align*}
    \sigma(M^+)-\sigma(M^-)-(\sigma(M^+_B)-\sigma(M^-_B))&=\sigma(M^+)-\sigma(M^+_B)\\
    &=\sign\left(\det(P)\det(M^+)\right)-\sign\left(\det(P)\det(M^+_B)\right)\\
    &=\sign(\det(P))\left(\sign(\det(M^+))-\sign(\det(M^+_B))\right).
\end{align*}

\begin{proof}[Proof of Theorem \ref{mainthmB}]
    On one hand we have $\sign(\det(P))\in\{-1,0,1\}$, and since the determinants of $M^+$ and $M^+_B$ are both non-zero, $\sign(\det(M^+))-\sign(\det(M^+_B))$ has to be either zero or $\pm2$, such that the above difference is either zero or $\pm2$.
    
    Overall, since the correction term does not change, this means that the equivariant signature changes by either $0$ or $\pm2$ under a type B move, and so $\sigmaeq(K)-2\leq\sigmaeq(K')\leq\sigmaeq(K)+2$.
\end{proof}

\section{Proof of Theorem \ref{mainthmC}}

We can now show that the difference of signatures in the type C case is also either zero or $\pm2$. First note that there are two ways of applying a type C move as given in Figure \ref{fig:typeCmoves}, depending on the sign of the resulting crossings.
\begin{figure}[h!]
    \centering
    \includegraphics[width=0.7\textwidth]{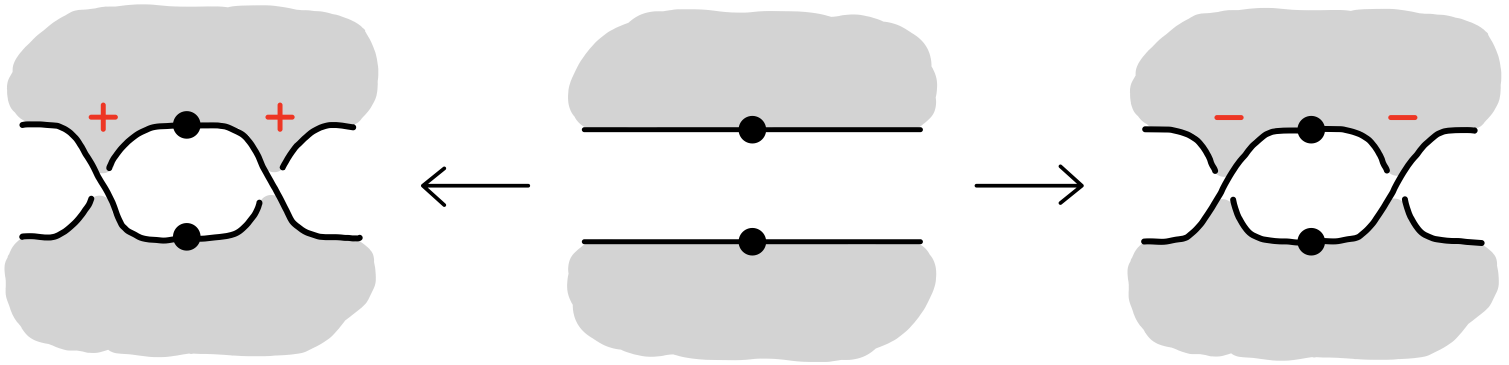}
    \caption{Two different type C moves: a positive move (left) and a negative move (right).}
    \label{fig:typeCmoves}
\end{figure}
We will focus on the positive case here, as the proof of the negative case proceeds in a similar fashion.

If $M$ is the $2n\times 2n$ Goeritz matrix before the type C move is applied, then the matrix $M_C$ after the move is an $(2n+2)\times(2n+2)$ matrix obtained from $M$ by adding two extra lines and columns as depicted here:

\[
M_C=\begin{bNiceArray}[first-row,first-col]{cccccccc}
    & a_1 & \cdots & a_n & a_{n+1} & a'_1 & \cdots & a'_n & a'_{n+1} \\
    a_1 & & & & & & & & \\
    \vdots & & & & v & & & & \\
    a_n & & & & & & & & \\
    a_{n+1} & & v^T & & 1-S & & & & \\
    a'_1 & & & & & & & & \\
    \vdots & & & & & & & & v \\
    a'_n & & & & & & & & \\
    a'_{n+1} & & & & & & v^T & & 1-S \\
\end{bNiceArray}
\]
where $v=\left[\sum_{k=1}^n a_1a_k,\ldots, \sum_{k=1}^n a_na_k\right]^T$ is a vector of size $n$, $S=\sum_{i=1}^n(a_i+a'_i)a_{n+1}$, and again the non-explicit coefficients are exactly the coefficients of $M$ (the term $a_i a_j$ is exactly $g_{ij}$, the $(i,j)$-entry of the Goertiz matrix). The matrices of the $+1$ and $-1$ eigenspaces are expressed as

\[
M^+_C=\begin{bNiceArray}[first-row,first-col]{ccccc}
    & a_1-a'_1 & \cdots & a_n-a'_n & a_{n+1}-a'_{n+1} \\
    a_1-a'_1 & & & & \\
    \vdots & & M^+ & & 2v \\
    a_n-a'_n & & & & \\
    a_{n+1}-a'_{n+1} & & 2v^T & & 2-2S \\
\end{bNiceArray},
\]
\[
M^-_C=\begin{bNiceArray}[first-row,first-col]{ccccc}
    & a_1+a'_1 & \cdots & a_n+a'_n & a_{n+1}+a'_{n+1} \\
    a_1+a'_1 & & & & \\
    \vdots & & M^- & & 2v \\
    a_n+a'_n & & & & \\
    a_{n+1}+a'_{n+1} & & 2v^T & & 2-2S \\
\end{bNiceArray}
\]

Consider now the two crossings created after the type C move, and consider their simultaneous resolution, choosing the resolution which always yields a split link; see Figure \ref{fig:typeCresolution} for an explicit diagram.
\begin{figure}[h!]
    \centering
    \includegraphics[width=0.45\textwidth]{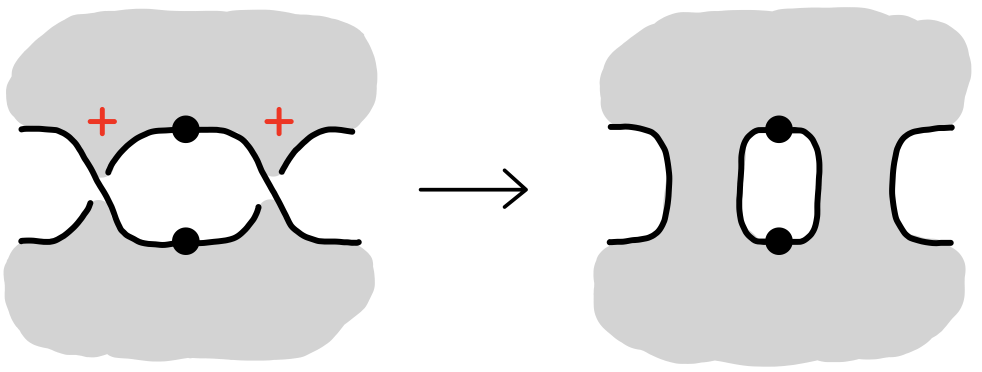}
    \caption{The resolution of the two new crossings, giving a split link.}
    \label{fig:typeCresolution}
\end{figure}

The Goeritz matrix $M_0$ of this link is given by
\[
M_0=\begin{bNiceArray}[first-row,first-col]{cccccccc}
    & a_1 & \cdots & a_n & a_{n+1} & a'_1 & \cdots & a'_n & a'_{n+1} \\
    a_1 & & & & & & & & \\
    \vdots & & & & v & & & & \\
    a_n & & & & & & & & \\
    a_{n+1} & & v^T & & -S & & & & \\
    a'_1 & & & & & & & & \\
    \vdots & & & & & & & & v \\
    a'_n & & & & & & & & \\
    a'_{n+1} & & & & & & v^T & & -S \\
\end{bNiceArray}
\]
such that we can write
\[
M_0^+=\begin{bNiceArray}[first-row,first-col]{ccccc}
    & a_1-a'_1 & \cdots & a_n-a'_n & a_{n+1}-a'_{n+1} \\
    a_1-a'_1 & & & & \\
    \vdots & & M^+ & & 2v \\
    a_n-a'_n & & & & \\
    a_{n+1}-a'_{n+1} & & 2v^T & & -2S \\
\end{bNiceArray},
\]
\[
M_0^-=\begin{bNiceArray}[first-row,first-col]{ccccc}
    & a_1+a'_1 & \cdots & a_n+a'_n & a_{n+1}+a'_{n+1} \\
    a_1+a'_1 & & & & \\
    \vdots & & M^- & & 2v \\
    a_n+a'_n & & & & \\
    a_{n+1}+a'_{n+1} & & 2v^T & & -2S \\
\end{bNiceArray}
\]

and so the determinants of $M_C^+$ and $M_C^-$ can be written as $\det(M_C^\pm)=\det(M_0^\pm)+2\det(M^\pm)$. We will make use of the following observation:
\begin{lem}
    With notations as above, we have $\det(M_0^-)=0$, and so $\det(M_C^-)=2\det(M^-)$.
\end{lem}
\begin{proof}
    Recall that the Goeritz matrix $G$ associated to a knot diagram is obtained from a matrix $G'$ by deleting the $i$-th row and $i$-th column. Note that $M_0$ is exactly $G'$, such that each of this matrix's columns sum up to zero, and so $\sum_{k=1}^{n+1}a_i(a_k+a'_k)=0$ for each $1\leq i\leq n+1$ -- note that this is also the case when replacing $a_i$ by $a'_i$.

    Now, we will focus on the last row of the matrix $M_0^-$, namely we will compute the coefficients $(a_i+a'_i)(a_{n+1}+a'_{n+1})$ for every $1\leq i\leq n+1$. If $i\neq n+1$, then
    \begin{align*}
        (a_i+a'_i)(a_{n+1}+a'_{n+1})&=a_ia_{n+1}+a_ia'_{n+1}+a'_ia'_{n+1}+a'_ia_{n+1}\\
        &=2a_i(a_{n+1}+a'_{n+1})=-2\sum_{k=1}^n a_i(a_k+a'_k)\\
        &=-\sum_{k=1}^n (a_i+a'_i)(a_k+a'_k)
    \end{align*}
    since $a_i(a_{n+1}+a'_{n+1})=-\sum_{k=1}^{n}a_i(a_k+a'_k)$ and $a_i(a_k+a'_k)=a'_i(a_k+a'_k)$ by symmetry of the diagram of $K$.

    As for the diagonal coefficient, we have
    \begin{align*}
        (a_{n+1}+a'_{n+1})^2&=2a^2_{n+1}+2a_{n+1}a'_{n+1}=2a_{n+1}(a_{n+1}+a'_{n+1})=-2\sum_{k=1}^n a_{n+1}(a_k+a'_k)\\
        &=-\sum_{k+1}^n (a_{n+1}+a'_{n+1})(a_k+a'_k)=-2\sum_{k=1}^n (a_{n+1}+a'_{n+1})a_k\\
        &=-2\sum_{k=1}^n \left(-\sum_{i=1}^n a_k(a_i+a'_i)\right)=2\sum_{k=1}^n \sum_{i=1}^n a_k(a_i+a'_i)\\
        &=\sum_{k=1}^n \sum_{i=1}^n (a_k+a'_k)(a_i+a'_i).
    \end{align*}

    In other words, we have found that the last row of $M_0^-$ is the negative of the sum of all the other rows, so its determinant must be zero.
\end{proof}

Going back to the signature computation, we first pick a $\sigma$-series $\Delta_1,\ldots,\Delta_n$ for $M^+$ and $\Gamma_1,\ldots,\Gamma_n$ for $M^-$. Since $M^+$, $M^-$, $M^+_C$ and $M^-_C$ are non-singular, by considering $\Delta_{n+1}=M^+_C$ and $\Gamma_{n+1}=M^-_C$, we find $\sigma$-series for $M^+_C$ and $M^-_C$. This means that

\begin{align*}
    \sigma(M^+_C)&= \sigma(M^+) + \sign\left(\det(M^+)\det(M^+_C)\right)\\
    \sigma(M^-_C)&= \sigma(M^-) + \sign\left(\det(M^-)\det(M^-_C)\right)
\end{align*}
and so
\begin{align*}
    &\sigma(M^+)-\sigma(M^-)-\left(\sigma(M^+_C)-\sigma(M^-_C)\right)\\
    =&\sign\left(\det(M^-)\det(M^-_C)\right)-\sign\left(\det(M^+)\det(M^+_C)\right)\\
    =&\sign\left(2\det(M^-)^2)\right)-\sign\left(\det(M^+)\det(M^+_C)\right)\\
    =&1-\sign\left(\det(M^+)\det(M^+_C)\right).
\end{align*}

\begin{proof}[Proof of Theorem \ref{mainthmC}]
    Since $M^+$ and $M_C^+$ have non-zero determinant, we find that the above quantity has to be either $0$ or $2$.
    
    Overall, since $e(D)$ either remains unchanged or changes by $-2$ in this case, this means that the equivariant signature changes by either $0$ or $2$ if the new crossings are unicolored, and changes by $0$ or $-2$ if they are bicolored. The same proof is valid for the negative type C move, except that in this case the equivariant signature changes by $0$ or $-2$ if the new crossings are unicolored, and by $0$ or $2$ if they are bicolored, and so overall, we find that $\sigmaeq(K)-2\leq\sigmaeq(K')\leq\sigmaeq(K)+2$.
\end{proof}

\section{Proof of Theorem \ref{mainthmA} in the case of a type A1 move}

Note that in the case of a type A move, either the crossing changes happen between a marked region, say $a_n$ (and $a'_n$), and the region containing the direction, in which case we will call this a type A1 move, or the crossing changes happen between two marked regions, say $a_{n-1}$ and $a_n$ (and the same for their symmetric counterparts), which we call a type A2 move. We will treat the two cases separately.

First for a type A1 move, if the original matrix is written as
\[
M=\begin{bNiceArray}[first-row,first-col]{ccccccc}
    & a_1 & \cdots & a_n & a'_1 & \cdots & a'_n \\
    a_1 & & & & & & \\
    \vdots & & & & & & \\
    a_n & & & \alpha & & & \beta \\
    a'_1 & & & & & & \\
    \vdots & & & & & & \\
    a'_n & & & \beta & & & \alpha \\
\end{bNiceArray},
\]
we obtain the following matrices for the $+1$ and $-1$ eigenspaces of the involution:
\[
M^+=\begin{bNiceArray}[first-row,first-col]{cccc}
    & a_1-a'_1 & \cdots & a_n-a'_n \\
    a_1-a'_1 & P & & \\
    \vdots & & & \\
    a_n-a'_n & & & 2(\alpha-\beta) \\
\end{bNiceArray},\; M^-=\begin{bNiceArray}[first-row,first-col]{cccc}
    & a_1+a'_1 & \cdots & a_n+a'_n \\
    a_1+a'_1 & Q & & \\
    \vdots & & & \\
    a_n+a'_n & & & 2(\alpha+\beta) \\
\end{bNiceArray},
\]
and the $2n\times 2n$ matrix $M_A$ associated to the diagram after a type A1 move can be given as

\[
M_A=\begin{bNiceArray}[first-row,first-col]{ccccccc}
    & a_1 & \cdots & a_n & a'_1 & \cdots & a'_n \\
    a_1 & & & & & & \\
    \vdots & & & & & & \\
    a_n & & & \alpha\pm2 & & & \beta \\
    a'_1 & & & & & & \\
    \vdots & & & & & & \\
    a'_n & & & \beta & & & \alpha\pm2 \\
\end{bNiceArray},
\]
such that
\[
M^+_A=\begin{bNiceArray}[first-row,first-col]{cccc}
    & a_1-a'_1 & \cdots & a_n-a'_n \\
    a_1-a'_1 & P & & \\
    \vdots & & & \\
    a_n-a'_n & & & 2(\alpha-\beta)\pm4 \\
\end{bNiceArray}, \; M^-_A=\begin{bNiceArray}[first-row,first-col]{cccc}
    & a_1+a'_1 & \cdots & a_n+a'_n \\
    a_1+a'_1 & Q & & \\
    \vdots & & & \\
    a_n+a'_n & & & 2(\alpha+\beta)\pm4 \\
\end{bNiceArray},
\]
where $P$ and $Q$ are two $(n-1)\times(n-1)$ symmetric matrices given as follows. Consider the two crossings which are flipped by the type A1 move. Consider their resolution which merges the unmarked region with the two regions $a_n$ and $a'_n$. Note that this basically amounts to undoing a type C move as in Figure \ref{fig:typeCmoves}.
If $M_1$ denotes the Goeritz matrix obtained after this resolution, then $M_1$ is obtained from $M$ by deleting the $n$-th and $2n$-th row and columns, such that $P=M_1^+$ and $Q=M_1^-$.

Since $P$ and $Q$ are nonsingular, we can pick a $\sigma$-series for each of them. If we pick $\Delta_1,\ldots,\Delta_{n-1}$ to be a $\sigma$-series for $P$ and $\Gamma_1,\ldots,\Gamma_{n-1}$ a $\sigma$-series for $Q$, we can pick $\Delta_n$ to be either $M^+$ or $M^+_A$ and $\Gamma_n$ to be either $M^-$ or $M^-_A$ to get $\sigma$-series for all the four matrices. We can now compute:
\begin{align*}
    &\sigma(M^+)-\sigma(M^-)-(\sigma(M^+_A)-\sigma(M^-_A))\\
    =&\sign\left(\det(P)\right)\sign\left(\det(M^+)\det(M^+_A)\right)-\sign\left(\det(Q)\right)\sign\left(\det(M^-)\det(M^-_A)\right)
\end{align*}
\begin{proof}[Proof of Theorem \ref{mainthmA} for type A1 moves]
    The four matrices have non-zero determinant, such that both terms of the subtraction are either $0$ or $\pm1$, and so the whole quantity is either $0$, $\pm1$ or $\pm2$.

Since the correction term changes by either $0$ or $\pm4$, we find that a type A1 move changes the signature by $0$, $\pm2$, $\pm4$, or $\pm6$.
\end{proof}

\section{Proof of Theorem \ref{mainthmA} in the case of a type A2 move}

For this type of unknotting move, and if the original Goeritz matrix is written as
\[
M=\begin{bNiceArray}[first-row,first-col]{ccccccccc}
    & a_1 & \cdots & a_{n-1} & a_n & a'_1 & \cdots & a'_{n-1} & a'_n \\
    a_1 & & & & & & & & \\
    \vdots & & & & & & & & \\
    a_{n-1} & & & \alpha & \gamma & & & \theta & \\
    a_n & & & \gamma & \beta & & & & \delta \\
    a'_1 & & & & & & & & \\
    \vdots & & & & & & & & \\
    a'_{n-1} & & & \theta & & & & \alpha & \gamma \\
    a'_n & & & & \delta & & & \gamma & \beta \\
\end{bNiceArray}
\]
then $M_A$ can be given as
\[
M_A=\begin{bNiceArray}[first-row,first-col]{ccccccccc}
    & a_1 & \cdots & a_{n-1} & a_n & a'_1 & \cdots & a'_{n-1} & a'_n \\
    a_1 & & & & & & & & \\
    \vdots & & & & & & & & \\
    a_{n-1} & & & \alpha\pm 2 & \gamma\mp 2 & & & \theta & \\
    a_n & & & \gamma\mp 2 & \beta\pm 2 & & & & \delta \\
    a'_1 & & & & & & & & \\
    \vdots & & & & & & & & \\
    a'_{n-1} & & & \theta & & & & \alpha\pm 2 & \gamma\mp 2 \\
    a'_n & & & & \delta & & & \gamma\mp 2 & \beta\pm 2 \\
\end{bNiceArray}
\]
depending on the sign of the crossing changes.

This means that the $+1$ and $-1$-eigenspaces can be given as
\[
M^+=\begin{bNiceArray}[first-row,first-col]{ccccc}
    & a_1-a'_1 & \cdots & a_{n-1}-a'_{n-1} & a_n-a'_n \\
    a_1-a'_1 &  & & & \\
    \vdots & & P & & \\
    a_{n-1}-a'_{n-1} & & & 2(\alpha-\theta) & 2\gamma \\
    a_n-a'_n & & & 2\gamma & 2(\beta-\delta) \\
\end{bNiceArray},
\]
\[
M^-=\begin{bNiceArray}[first-row,first-col]{ccccc}
    & a_1+a'_1 & \cdots & a_{n-1}+a'_{n-1} & a_n+a'_n \\
    a_1+a'_1 &  & & & \\
    \vdots & & Q & & \\
    a_{n-1}+a'_{n-1} & & & 2(\alpha+\theta) & 2\gamma \\
    a_n+a'_n & & & 2\gamma & 2(\beta+\delta) \\
\end{bNiceArray}
\]
and
\[
M^+_A=\begin{bNiceArray}[first-row,first-col]{ccccc}
    & a_1-a'_1 & \cdots & a_{n-1}-a'_{n-1} & a_n-a'_n \\
    a_1-a'_1 &  & & & \\
    \vdots & & P & & \\
    a_{n-1}-a'_{n-1} & & & 2(\alpha-\theta)\pm 4 & 2\gamma\mp 4 \\
    a_n-a'_n & & & 2\gamma\mp 4 & 2(\beta-\delta)\pm 4 \\
\end{bNiceArray},
\]
\[
M^-_A=\begin{bNiceArray}[first-row,first-col]{ccccc}
    & a_1-a'_1 & \cdots & a_{n-1}-a'_{n-1} & a_n-a'_n \\
    a_1-a'_1 &  & & & \\
    \vdots & & Q & & \\
    a_{n-1}-a'_{n-1} & & & 2(\alpha+\theta)\pm 4 & 2\gamma\mp 4 \\
    a_n-a'_n & & & 2\gamma\mp 4 & 2(\beta+\delta)\pm 4 \\
\end{bNiceArray}
\]
where $P$ and $Q$ are two $(n-2)\times(n-2)$ symmetric matrices.
\begin{remark}
    Notice that the matrices satisfy $M_A^+=M^+ \pm4uu^T$ and $M_A^-=M^- \pm4uu^T$, where $u$ is a vector of size $n$ equal to $e_{n-1}-e_n$. Additionally, the matrix $4uu^T$ is of rank $1$.
\end{remark}
We can see that this is the first case where more than one coefficient is modified, such that we cannot apply our usual method for this case. However, since we are dealing with a rank-one perturbation, we can still estimate how the signature changes.

Let us consider the specific case where the type A2 move changes two negative crossings to positive (where the orientation considered is the one induced by the black regions). This means that we can write $M_A^+=M^++4uu^T$, where $u=e_{n-1}-e_n$ (and the same for $M^-$). We can use a consequence of Weyl's inequality to understand how the eigenvalues of $M_A^+$ are related to those of $M^+$ (see for example \cite[Theorem 4.3.1]{HJ13}):

\begin{prop}
    Let $A,B$ be two $n\times n$ Hermitian matrices. Then:
    \begin{gather*}
        \lambda_i(A+B)\leq\lambda_{i+j}(A)+\lambda_{n-j}(B), \qquad j=0,1,\ldots,n-i\\
        \lambda_{i-j+1}(A)+\lambda_j(B)\leq\lambda_i(A+B), \qquad j=1,\ldots,i
    \end{gather*}
    for each $i=1,\ldots,n$, where $\{\lambda_i(A)\}_{i=1}^n$,$\{\lambda_i(B)\}_{i=1}^n$, and $\{\lambda_i(A+B)\}_{i=1}^n$ denote the set of eigenvalues of $A$, $B$, and $A+B$, arranged in algebraically nondecreasing order.
\end{prop}

In our case, $A=M^+$ and $B=4uu^T$, and the set of eigenvalues of $4uu^T$ is given by $\lambda_i(4uu^T)=0$ for $1\leq i\leq n-1$, and $\lambda_n(4uu^T)=8$. Therefore, this proposition restates as:
\[
\lambda_i(M^+)\leq\lambda_i(M^++4uu^T)\leq\lambda_{i+1}(M^+), \qquad i=1,\dots,n
\]
and so only one eigenvalue can cross zero as we move from $M^+$ to $M^++4uu^T$.

Now, the perturbation $4uu^T$ is positive semidefinite, which means that the positive index cannot decrease. One way to see this is by considering the bilinear forms $q(x)=x^TM^+x$ and $q_A(x)=x^TM_A^+x=q(x)+4(x^Tu)^2$: if $x$ is a direction where $q(x)>0$, then $q_A(x)\geq q(x)>0$ and so $x$ is also a positive direction for $q_A$, \textit{i.e.} no positive direction can become negative.

\begin{proof}[Proof of Theorem \ref{mainthmA} for type A2 moves]
In summary, we have found that either no eigenvalue crosses zero, and so the signature does not change, or there is exactly one eigenvalue crossing zero, where in that case the positive index increases by one.
More precisely, we can write the determinant as:
\begin{align*}
    \det(M^++4uu^T)&=\det(M^+(I_n+4(M^+)^{-1}uu^T))\\
    &=\det(M^+)\det(I_n+4(M^+)^{-1}uu^T)\\
    &=\det(M^+)\det(I_n+4u^T(M^+)^{-1}u)\\
    &=\det(M^+)(1+4u^T(M^+)^{-1}u)
\end{align*}
since $\det(I_m+AB)=\det(I_k+BA)$ for any two matrices $A$ and $B$ of sizes respectively $m\times k$ and $k\times m$. With this expression, the number of positive eigenvalues of $M_A^+$ is equal to $k+1$ if $1+4u^T(M^+)^{-1}u<0$, and to $k$ if $1+4u^T(M^+)^{-1}u>0$ (note that the case zero does not happen since $M_A^+$ is nonsingular).

Now, the same argument can be applied to $M^-$, and so the difference of signatures satisfies $\lvert\sigma(M^+)-\sigma(M^-)-(\sigma(M^+_A)-\sigma(M^-_A))\rvert\leq 2$. Since the correction term changes by at most $4$, a type A2 move changes the signature by at most $6$.

Note that the same applies when the type A2 move changes two positive crossings to negative (so when $M_A^+=M^+-4uu^T$), because in that case we can write $M^+=M_A^++4uu^T$ and apply the same argument as before.
This concludes the proof of Theorem \ref{mainthmA}.
\end{proof}

\section{Examples}

In this section, we will give examples of type A, B, and C moves on some given knots, which either do not change the equivariant signature, or change it by the maximal possible quantity in each case.

\begin{exa}[The type A bound is tight]
    Consider the admissible diagram for the knot $9_{40}$ given in Figure \ref{fig:9-40}.

\begin{figure}[ht!]
    \def\svgwidth{0,3\columnwidth}
    \centering
    %
    %% Creator: Inkscape 1.2.2 (732a01da63, 2022-12-09), www.inkscape.org
%% Accompanies image file '9-40.pdf' (pdf, eps, ps)
\begingroup%
  \makeatletter%
  \providecommand\color[2][]{%
    \errmessage{(Inkscape) Color is used for the text in Inkscape, but the package 'color.sty' is not loaded}%
    \renewcommand\color[2][]{}%
  }%
  \providecommand\transparent[1]{%
    \errmessage{(Inkscape) Transparency is used (non-zero) for the text in Inkscape, but the package 'transparent.sty' is not loaded}%
    \renewcommand\transparent[1]{}%
  }%
  \providecommand\rotatebox[2]{#2}%
  \newcommand*\fsize{\dimexpr\f@size pt\relax}%
  \newcommand*\lineheight[1]{\fontsize{\fsize}{#1\fsize}\selectfont}%
  \ifx\svgwidth\undefined%
    \setlength{\unitlength}{438.18784044bp}%
    \ifx\svgscale\undefined%
      \relax%
    \else%
      \setlength{\unitlength}{\unitlength * \real{\svgscale}}%
    \fi%
  \else%
    \setlength{\unitlength}{\svgwidth}%
  \fi%
  \global\let\svgwidth\undefined%
  \global\let\svgscale\undefined%
  \makeatother%
  \begin{picture}(1,1.17664276)%
    \lineheight{1}%
    \setlength\tabcolsep{0pt}%
    \put(0,0){\includegraphics[width=\unitlength,page=1]{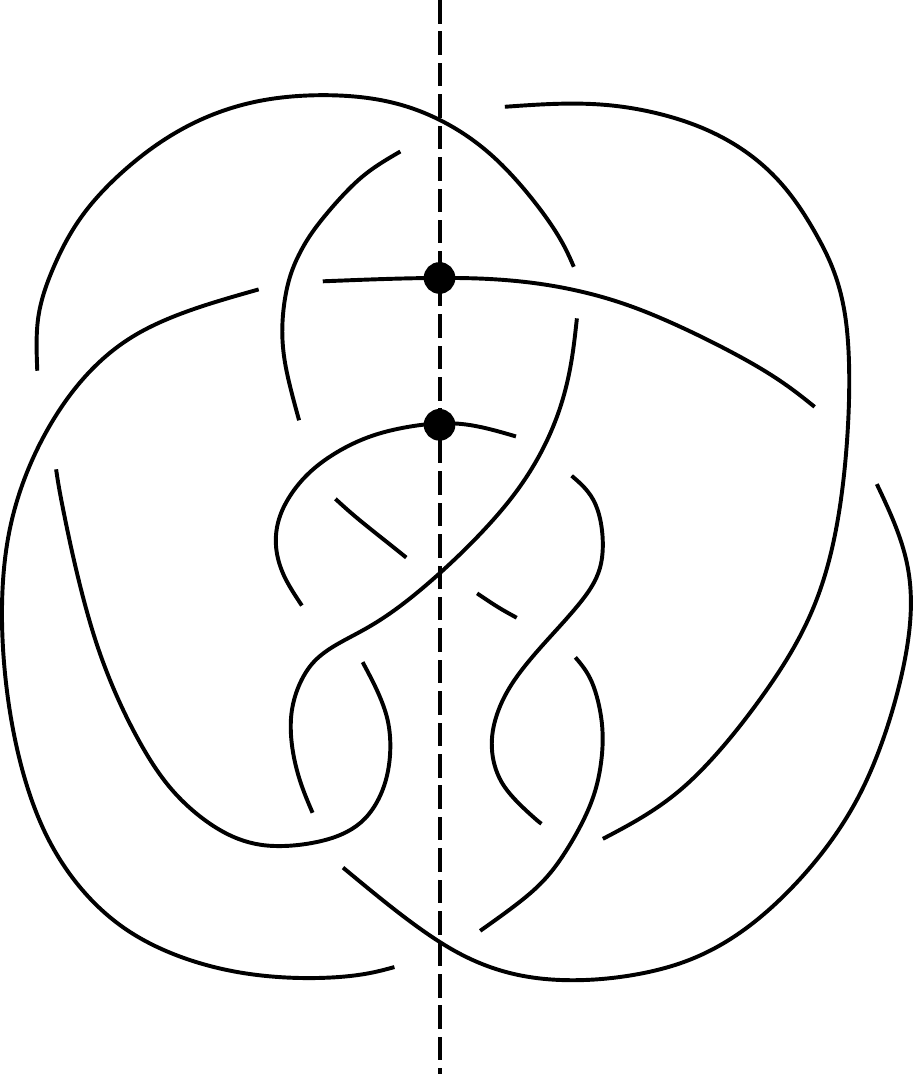}}%
    \put(0.17251957,0.1851126){\makebox(0,0)[lt]{\lineheight{1.25}\smash{\begin{tabular}[t]{l}$a$\end{tabular}}}}%
    \put(0.69626807,0.1851126){\makebox(0,0)[lt]{\lineheight{1.25}\smash{\begin{tabular}[t]{l}$a'$\end{tabular}}}}%
    \put(0.3505255,0.33573329){\makebox(0,0)[lt]{\lineheight{1.25}\smash{\begin{tabular}[t]{l}$b$\end{tabular}}}}%
    \put(0.57303292,0.33573329){\makebox(0,0)[lt]{\lineheight{1.25}\smash{\begin{tabular}[t]{l}$b'$\end{tabular}}}}%
    \put(0.34710231,0.53770157){\makebox(0,0)[lt]{\lineheight{1.25}\smash{\begin{tabular}[t]{l}$c$\end{tabular}}}}%
    \put(0.55934019,0.53427838){\makebox(0,0)[lt]{\lineheight{1.25}\smash{\begin{tabular}[t]{l}$c'$\end{tabular}}}}%
    \put(0.17594276,0.89371343){\makebox(0,0)[lt]{\lineheight{1.25}\smash{\begin{tabular}[t]{l}$d$\end{tabular}}}}%
    \put(0.70653763,0.89371343){\makebox(0,0)[lt]{\lineheight{1.25}\smash{\begin{tabular}[t]{l}$d'$\end{tabular}}}}%
  \end{picture}%
\endgroup%

    \caption{An admissible diagram for the knot $9_{40}$.}
    \label{fig:9-40}
\end{figure}
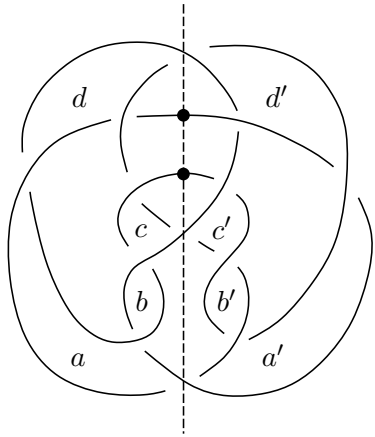

    The Goeritz matrix is given by
    \[
M=\begin{bNiceArray}[first-row,first-col]{cccccccc}
    & a & b & c & d & a' & b' & c' & d' \\
    a & -3 & 1 & 0 & 1 & 1 & 0 & 0 & 0 \\
    b & 1 & -2 & 1 & 0 & 0 & 0 & 0 & 0 \\
    c & 0 & 1 & -1 & 0 & 0 & 0 & -1 & 0 \\
    d & 1 & 0 & 0 & -3 & 0 & 0 & 0 & 1 \\
    a' & 1 & 0 & 0 & 0 & -3 & 1 & 0 & 1 \\
    b' & 0 & 0 & 0 & 0 & 1 & -2 & 1 & 0 \\
    c' & 0 & 0 & -1 & 0 & 0 & 1 & -1 & 0 \\
    d' & 0 & 0 & 0 & 1 & 1 & 0 & 0 & -3 \\
\end{bNiceArray}
\]
and
\[
M^+=\begin{bNiceArray}[first-row,first-col]{ccccc}
    & a-a' & b-b' & c-c' & d-d' \\
    a-a' & -8 & 2 & 0 & 2 \\
    b-b' & 2 & -4 & 2 & 0 \\
    c-c' & 0 & 2 & 0 & 0 \\
    d-d' & 2 & 0 & 0 & -8 \\
\end{bNiceArray},\,
M^-=\begin{bNiceArray}[first-row,first-col]{ccccc}
    & a+a' & b+b' & c+c' & d+d' \\
    a+a' & -4 & 2 & 0 & 2 \\
    b+b' & 2 & -4 & 2 & 0 \\
    c+c' & 0 & 2 & -4 & 0 \\
    d+d' & 2 & 0 & 0 & -4 \\
\end{bNiceArray}.
\]
In this case, we have $\sigma(M^+)=-2$, $\sigma(M^-)=-4$, and $e(D)=-4$, such that the equivariant signature is $\sigmaeq(9_{40}b^-)=-2-(-4)-(-4)=6$.

We perform two consecutive type A moves as given by Figure \ref{fig:9-40typeA}, and denote by $M_1$ (resp. $M_2$) the matrix obtained after performing one (resp. two) type A moves.

\begin{figure}[ht!]
    \def\svgwidth{0,9\columnwidth}
    \centering
    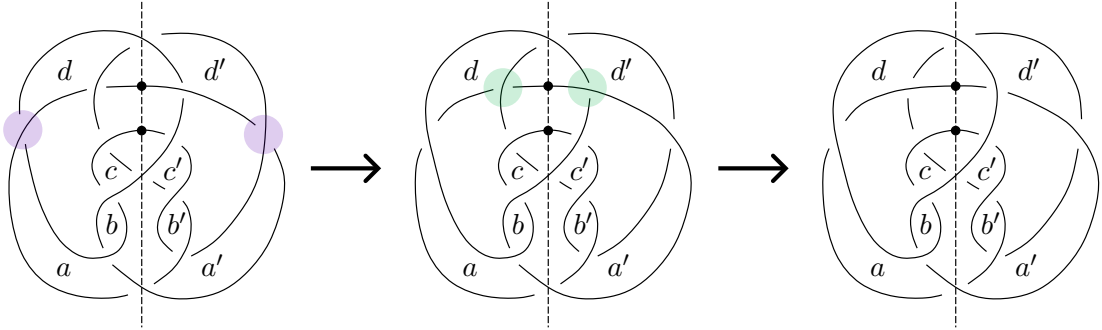

    \caption{Two consecutive type A moves applied on the highlighted crossings of the diagram of $9_{40}$ (left), giving first $m(3_1)\#m(3_1)$ (center) and then the unknot (right).}
    \label{fig:9-40typeA}
\end{figure}

After one type A move we obtain the matrices
\[
M_1=\begin{bNiceArray}[first-row,first-col]{cccccccc}
    & a & b & c & d & a' & b' & c' & d' \\
    a & -1 & 1 & 0 & -1 & 1 & 0 & 0 & 0 \\
    b & 1 & -2 & 1 & 0 & 0 & 0 & 0 & 0 \\
    c & 0 & 1 & -1 & 0 & 0 & 0 & -1 & 0 \\
    d & -1 & 0 & 0 & -1 & 0 & 0 & 0 & 1 \\
    a' & 1 & 0 & 0 & 0 & -1 & 1 & 0 & -1 \\
    b' & 0 & 0 & 0 & 0 & 1 & -2 & 1 & 0 \\
    c' & 0 & 0 & -1 & 0 & 0 & 1 & -1 & 0 \\
    d' & 0 & 0 & 0 & 1 & -1 & 0 & 0 & -1 \\
\end{bNiceArray}
\]
and
\[
M_1^+=\begin{bNiceArray}[first-row,first-col]{ccccc}
    & a-a' & b-b' & c-c' & d-d' \\
    a-a' & -4 & 2 & 0 & -2 \\
    b-b' & 2 & -4 & 2 & 0 \\
    c-c' & 0 & 2 & 0 & 0 \\
    d-d' & -2 & 0 & 0 & -4 \\
\end{bNiceArray},\,
M_1^-=\begin{bNiceArray}[first-row,first-col]{ccccc}
    & a+a' & b+b' & c+c' & d+d' \\
    a+a' & 0 & 2 & 0 & -2 \\
    b+b' & 2 & -4 & 2 & 0 \\
    c+c' & 0 & 2 & -4 & 0 \\
    d+d' & -2 & 0 & 0 & 0 \\
\end{bNiceArray}.
\]

This yields $\sigma(M_1^+)=-2$, $\sigma(M_1^-)=-2$, and $e(D)=0$, such that $\sigmaeq(3_1\#r3_1)=0$, the connected sum of the trefoil with its reverse. Finally, after the second type A move, we obtain

\[
M_2=\begin{bNiceArray}[first-row,first-col]{cccccccc}
    & a & b & c & d & a' & b' & c' & d' \\
    a & -1 & 1 & 0 & -1 & 1 & 0 & 0 & 0 \\
    b & 1 & -2 & 1 & 0 & 0 & 0 & 0 & 0 \\
    c & 0 & 1 & -1 & 0 & 0 & 0 & -1 & 0 \\
    d & -1 & 0 & 0 & 1 & 0 & 0 & 0 & 1 \\
    a' & 1 & 0 & 0 & 0 & -1 & 1 & 0 & -1 \\
    b' & 0 & 0 & 0 & 0 & 1 & -2 & 1 & 0 \\
    c' & 0 & 0 & -1 & 0 & 0 & 1 & -1 & 0 \\
    d' & 0 & 0 & 0 & 1 & -1 & 0 & 0 & 1 \\
\end{bNiceArray}
\]
and
\[
M_2^+=\begin{bNiceArray}[first-row,first-col]{ccccc}
    & a-a' & b-b' & c-c' & d-d' \\
    a-a' & -4 & 2 & 0 & -2 \\
    b-b' & 2 & -4 & 2 & 0 \\
    c-c' & 0 & 2 & 0 & 0 \\
    d-d' & -2 & 0 & 0 & 0 \\
\end{bNiceArray},\,
M_2^-=\begin{bNiceArray}[first-row,first-col]{ccccc}
    & a+a' & b+b' & c+c' & d+d' \\
    a+a' & 0 & 2 & 0 & -2 \\
    b+b' & 2 & -4 & 2 & 0 \\
    c+c' & 0 & 2 & -4 & 0 \\
    d+d' & -2 & 0 & 0 & 4 \\
\end{bNiceArray}.
\]

We find that $\sigma(M_1^+)=0$, $\sigma(M_1^-)=0$, and $e(D)=0$, so $\sigmaeq(U)=0$.
\end{exa}

\begin{exa}[The type B bound is tight]
    We consider now a type B move performed on the following diagram of the knot $5_1$:

\begin{figure}[ht!]
    \def\svgwidth{0,55\columnwidth}
    \centering
    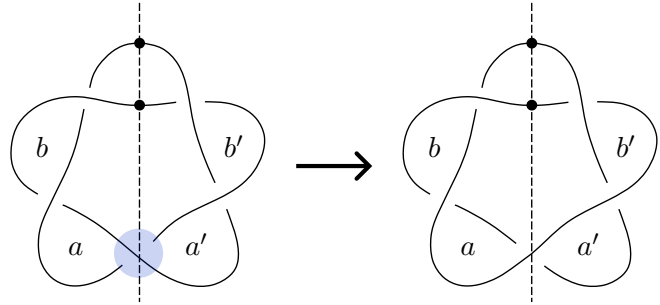

    \caption{A type B move performed on the highlighted crossing of the diagram of $5_1$.}
    \label{fig:5-1typeB}
\end{figure}

The Goeritz matrices before and after the type B move are
\[
M=\begin{bNiceArray}[first-row,first-col]{ccccc}
    & a & b & a' & b' \\
    a & -2 & 1 & 1 & 0 \\
    b & 1 & -2 & 0 & 0 \\
    a' & 1 & 0 & -2 & 1 \\
    b' & 0 & 0 & 1 & -2 \\
\end{bNiceArray},\,
M_B=\begin{bNiceArray}[first-row,first-col]{ccccc}
    & a & b & a' & b' \\
    a & 0 & 1 & -1 & 0 \\
    b & 1 & -2 & 0 & 0 \\
    a' & -1 & 0 & 0 & 1 \\
    b' & 0 & 0 & 1 & -2 \\
\end{bNiceArray}.
\]
They yield the following matrices for the respective $+1$ and $-1$-eigenspaces:
\[
M^+=\begin{bNiceArray}[first-row,first-col]{ccc}
    & a-a' & b-b' \\
    a-a' & -6 & 2 \\
    b-b' & 2 & -4 \\
\end{bNiceArray},\quad
M^-=\begin{bNiceArray}[first-row,first-col]{ccc}
    & a+a' & b+b' \\
    a+a' & -2 & 2 \\
    b+b' & 2 & -4 \\
\end{bNiceArray}
\]
and
\[
M_B^+=\begin{bNiceArray}[first-row,first-col]{ccc}
    & a-a' & b-b' \\
    a-a' & 2 & 2 \\
    b-b' & 2 & -4 \\
\end{bNiceArray},\quad
M_B^-=\begin{bNiceArray}[first-row,first-col]{ccc}
    & a+a' & b+b' \\
    a+a' & -2 & 2 \\
    b+b' & 2 & -4 \\
\end{bNiceArray}.
\]
The correction term for both diagrams is equal to $4$. This means that the signature before the move is $\sigmaeq(5_1)=-2-(-2)-4=-4$, and after the move it is $\sigmaeq(3_1)=0-(-2)-4=-2$.
\end{exa}

\begin{exa}[The type C bound is tight]
    Continuing with the knot $5_1$, we now perform a type C move on the original diagram:
\begin{figure}[ht!]
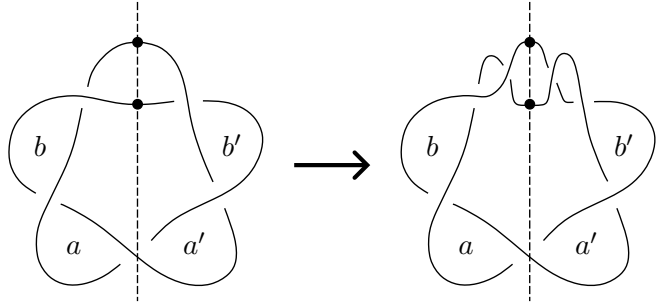

    \def\svgwidth{0,55\columnwidth}
    \centering
    

    \caption{A type C move performed on the diagram of $5_1$.}
    \label{fig:5-1typeC}
\end{figure}
The Goeritz matrix after the type C move is
\[
M_C=\begin{bNiceArray}[first-row,first-col]{ccccccc}
    & a & b & c & a' & b' & c' \\
    a & -2 & 1 & 0 & 1 & 0 & 0 \\
    b & 1 & -2 & 1 & 0 & 0 & 0 \\
    c & 0 & 1 & 0 & 0 & 0 & 0 \\
    a' & 1 & 0 & 0 & -2 & 1 & 0 \\
    b' & 0 & 0 & 0 & 1 & -2 & 1 \\
    c' & 0 & 0 & 0 & 0 & 1 & 0 \\
\end{bNiceArray},
\]
and it gives the following matrices for the respective $+1$ and $-1$-eigenspaces:
\[
M_C^+=\begin{bNiceArray}[first-row,first-col]{cccc}
    & a-a' & b-b' & c-c' \\
    a-a' & -6 & 2 & 0 \\
    b-b' & 2 & -4 & 2 \\
    c-c' & 0 & 2 & 0 \\
\end{bNiceArray},\quad
M_C^-=\begin{bNiceArray}[first-row,first-col]{cccc}
    & a+a' & b+b' & c+c' \\
    a+a' & -2 & 2 & 0 \\
    b+b' & 2 & -4 & 2 \\
    c+c' & 0 & 2 & 0 \\
\end{bNiceArray}.
\]
The correction term after the type C move is now $2$, such that the signature after the move is $-1-(-1)-2=-2$. Note that after the move, we obtain the diagram of the right-handed trefoil.
\end{exa}

\begin{exa}
    Consider the same diagram of $9_{40}$ as before; we can perform a type C move to obtain the diagram given in Figure \ref{fig:9-40typeC}.
\begin{figure}[ht!]
    \def\svgwidth{0,3\columnwidth}
    \centering
    %
    %% Creator: Inkscape 1.2.2 (732a01da63, 2022-12-09), www.inkscape.org
%% Accompanies image file '9-40typeC.pdf' (pdf, eps, ps)
\begingroup%
  \makeatletter%
  \providecommand\color[2][]{%
    \errmessage{(Inkscape) Color is used for the text in Inkscape, but the package 'color.sty' is not loaded}%
    \renewcommand\color[2][]{}%
  }%
  \providecommand\transparent[1]{%
    \errmessage{(Inkscape) Transparency is used (non-zero) for the text in Inkscape, but the package 'transparent.sty' is not loaded}%
    \renewcommand\transparent[1]{}%
  }%
  \providecommand\rotatebox[2]{#2}%
  \newcommand*\fsize{\dimexpr\f@size pt\relax}%
  \newcommand*\lineheight[1]{\fontsize{\fsize}{#1\fsize}\selectfont}%
  \ifx\svgwidth\undefined%
    \setlength{\unitlength}{438.18784044bp}%
    \ifx\svgscale\undefined%
      \relax%
    \else%
      \setlength{\unitlength}{\unitlength * \real{\svgscale}}%
    \fi%
  \else%
    \setlength{\unitlength}{\svgwidth}%
  \fi%
  \global\let\svgwidth\undefined%
  \global\let\svgscale\undefined%
  \makeatother%
  \begin{picture}(1,1.17664276)%
    \lineheight{1}%
    \setlength\tabcolsep{0pt}%
    \put(0,0){\includegraphics[width=\unitlength,page=1]{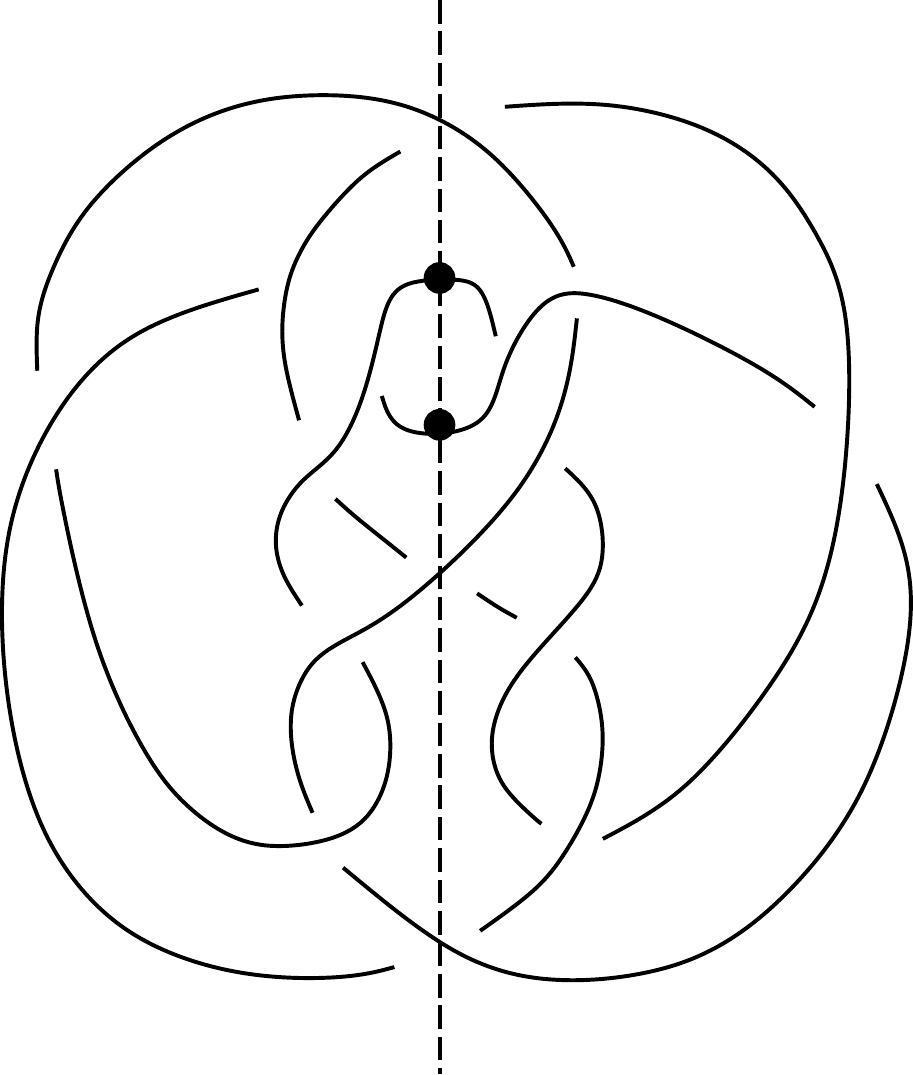}}%
    \put(0.17251957,0.1851126){\makebox(0,0)[lt]{\lineheight{1.25}\smash{\begin{tabular}[t]{l}$a$\end{tabular}}}}%
    \put(0.69626807,0.1851126){\makebox(0,0)[lt]{\lineheight{1.25}\smash{\begin{tabular}[t]{l}$a'$\end{tabular}}}}%
    \put(0.3505255,0.33573329){\makebox(0,0)[lt]{\lineheight{1.25}\smash{\begin{tabular}[t]{l}$b$\end{tabular}}}}%
    \put(0.57303292,0.33573329){\makebox(0,0)[lt]{\lineheight{1.25}\smash{\begin{tabular}[t]{l}$b'$\end{tabular}}}}%
    \put(0.34710231,0.53770157){\makebox(0,0)[lt]{\lineheight{1.25}\smash{\begin{tabular}[t]{l}$c$\end{tabular}}}}%
    \put(0.55934019,0.53427838){\makebox(0,0)[lt]{\lineheight{1.25}\smash{\begin{tabular}[t]{l}$c'$\end{tabular}}}}%
    \put(0.17594276,0.89371343){\makebox(0,0)[lt]{\lineheight{1.25}\smash{\begin{tabular}[t]{l}$d$\end{tabular}}}}%
    \put(0.70653763,0.89371343){\makebox(0,0)[lt]{\lineheight{1.25}\smash{\begin{tabular}[t]{l}$d'$\end{tabular}}}}%
    \put(0,0){\includegraphics[width=\unitlength,page=2]{9-40typeC.pdf}}%
  \end{picture}%
\endgroup%

    \caption{A type C move performed on the diagram of $9_{40}$.}
    \label{fig:9-40typeC}
\end{figure}

    The Goeritz matrix is now given by
    \[
M_C=\begin{bNiceArray}[first-row,first-col]{cccccccccc}
    & a & b & c & d & e & a' & b' & c' & d' & e' \\
    a & -3 & 1 & 0 & 1 & 0 & 1 & 0 & 0 & 0 & 0 \\
    b & 1 & -2 & 1 & 0 & 0 & 0 & 0 & 0 & 0 & 0 \\
    c & 0 & 1 & -1 & 0 & 1 & 0 & 0 & -1 & 0 & 0 \\
    d & 1 & 0 & 0 & -3 & 1 & 0 & 0 & 0 & 1 & 0 \\
    e & 0 & 0 & 1 & 1 & -1 & 0 & 0 & 0 & 0 & 0 \\
    a' & 1 & 0 & 0 & 0 & 0 & -3 & 1 & 0 & 1 & 0 \\
    b' & 0 & 0 & 0 & 0 & 0 & 1 & -2 & 1 & 0 & 0 \\
    c' & 0 & 0 & -1 & 0 & 0 & 0 & 1 & -1 & 0 & 1 \\
    d' & 0 & 0 & 0 & 1 & 0 & 1 & 0 & 0 & -3 & 1 \\
    e' & 0 & 0 & 0 & 0 & 0 & 0 & 0 & 1 & 1 & -1 \\
\end{bNiceArray}
\]
and
\[
M_C^+=\begin{bNiceArray}[first-row,first-col]{cccccc}
    & a-a' & b-b' & c-c' & d-d' & e-e' \\
    a-a' & -8 & 2 & 0 & 2 & 0 \\
    b-b' & 2 & -4 & 2 & 0 & 0 \\
    c-c' & 0 & 2 & 0 & 0 & 2 \\
    d-d' & 2 & 0 & 0 & -8 & 2 \\
    e-e' & 0 & 0 & 2 & 2 & -2 \\
\end{bNiceArray},\]\[
M_C^-=\begin{bNiceArray}[first-row,first-col]{cccccc}
    & a+a' & b+b' & c+c' & d+d' & e+e' \\
    a+a' & -4 & 2 & 0 & 2 & 0 \\
    b+b' & 2 & -4 & 2 & 0 & 0 \\
    c+c' & 0 & 2 & -4 & 0 & 2 \\
    d+d' & 2 & 0 & 0 & -4 & 2 \\
    e+e' & 0 & 0 & 2 & 2 & -2 \\
\end{bNiceArray}.
\]

We get $\sigma(M_C^+)=-3$, $\sigma(M_C^-)=-3$, and $e(D)=-6$, such that the equivariant signature is $6$. Note that in this case, performing a type C move modified the difference of signatures as well as the correction term, but overall the equivariant signature remained the same.
\end{exa}

\nocite{*}
\bibliographystyle{plain}
\bibliography{references}

\bigskip
\footnotesize

\textit{E-mail address}: \texttt{zampa.sarah@renyi.hu}

\bigskip
\textsc{Department of Geometry and Algebra, Institute of Mathematics, Budapest University of Technology and Economics, Műegyetem rkp. 3., H-1111 Budapest, Hungary}\par
\textsc{Department of Algebraic Geometry and Differential Topology, HUN-REN Alfréd R\'{e}nyi Institute of Mathematics, Reáltanoda u. 13-15., H-1053 Budapest, Hungary}

\end{document}